\documentclass[final,12pt]{elsarticle}
\usepackage{srcltx}
\usepackage{eurosym}
\usepackage{mathtools}
\usepackage{amsmath}
\usepackage{cases}
\usepackage{amsfonts}
\usepackage{amssymb}
\usepackage{amsthm}
\usepackage{graphicx}
\usepackage{mathrsfs}
\usepackage{xcolor}
\usepackage{xpatch}
\usepackage{exscale}
\usepackage{latexsym}
\usepackage{tikz}
\usepackage{caption}
\xpatchcmd{\MaketitleBox}{\hrule}{}{}{}
\xpatchcmd{\MaketitleBox}{\hrule}{}{}{}
\usepackage[colorlinks,plainpages=true,pdfpagelabels,hypertexnames=true,colorlinks=true,pdfstartview=FitV,linkcolor=blue,citecolor=red,urlcolor=black]{hyperref}
\PassOptionsToPackage{unicode}{hyperref}
\PassOptionsToPackage{naturalnames}{hyperref}
\usepackage{enumerate}
\usepackage[shortlabels]{enumitem}
\usepackage{bookmark}
\usepackage{wasysym}
\usepackage{esint}
\usepackage[ddmmyyyy]{datetime}
\usepackage[margin=2cm]{geometry}
\makeatletter
\g@addto@macro\normalsize{%
  \setlength\abovedisplayskip{4pt}
  \setlength\belowdisplayskip{4pt}
  \setlength\abovedisplayshortskip{4pt}
  \setlength\belowdisplayshortskip{4pt}
}
\numberwithin{equation}{section}
\everymath{\displaystyle}
\usepackage[capitalize,nameinlink]{cleveref}
\crefname{section}{Section}{Sections}
\crefname{subsection}{Subsection}{Subsections}
\crefname{condition}{Condition}{Conditions}
\crefname{hypothesis}{Hypothesis}{Conditions}
\crefname{assumption}{Assumption}{Assumptions}
\crefname{lemma}{Lemma}{Lemmas}
\crefname{definition}{Definition}{Definitions}

\crefformat{equation}{\textup{#2(#1)#3}}
\crefrangeformat{equation}{\textup{#3(#1)#4--#5(#2)#6}}
\crefmultiformat{equation}{\textup{#2(#1)#3}}{ and \textup{#2(#1)#3}}
{, \textup{#2(#1)#3}}{, and \textup{#2(#1)#3}}
\crefrangemultiformat{equation}{\textup{#3(#1)#4--#5(#2)#6}}%
{ and \textup{#3(#1)#4--#5(#2)#6}}{, \textup{#3(#1)#4--#5(#2)#6}}%
{, and \textup{#3(#1)#4--#5(#2)#6}}

\Crefformat{equation}{#2Equation~\textup{(#1)}#3}
\Crefrangeformat{equation}{Equations~\textup{#3(#1)#4--#5(#2)#6}}
\Crefmultiformat{equation}{Equations~\textup{#2(#1)#3}}{ and \textup{#2(#1)#3}}
{, \textup{#2(#1)#3}}{, and \textup{#2(#1)#3}}
\Crefrangemultiformat{equation}{Equations~\textup{#3(#1)#4--#5(#2)#6}}%
{ and \textup{#3(#1)#4--#5(#2)#6}}{, \textup{#3(#1)#4--#5(#2)#6}}%
{, and \textup{#3(#1)#4--#5(#2)#6}}

\crefdefaultlabelformat{#2\textup{#1}#3}
\numberwithin{equation}{section}

\newtheorem{theorem} {Theorem}[section]

\newtheorem{lemma}{Lemma}[section]

\newtheorem{example}{Example}[section]
\newtheorem{counter example}{Counter Example}[section]
\newtheorem{remark} {Remark}[section]
\newtheorem{definition} {Definition}[section]

\newtheorem{assumption}{Assumption}[section]

\newenvironment{Definitions}
{%

\begin{enumerate}}%
{\end{enumerate}}

\newenvironment{Assumptions}
{%

\begin{enumerate}}%
{\end{enumerate}}

\def\CC{{\rm \kern.24em \vrule width.02em height1.4ex depth-.05ex \kern-.26emC}}

\def\TagOnRight

\def\AA{{it I} \hskip-3pt{\tt A}}

\def\QQ{\rlap {\raise 0.4ex \hbox{$\scriptscriptstyle |$}} {\hskip -0.1em Q}}

\makeatletter
\newcommand{\vo}{\vec{o}\@ifnextchar{^}{\,}{}}
\makeatother

\def\YYint#1#2#3{{\setbox0=\hbox{$#1{#2#3}{\iint}$}
    \vcenter{\hbox{$#2#3$}}\kern-.50\wd0}}

\def\XXint#1#2#3{{\setbox0=\hbox{$#1{#2#3}{\int}$}
    \vcenter{\hbox{$#2#3$}}\kern-.50\wd0}}

\makeatletter
\def\namedlabel#1#2{\begingroup
   \def\@currentlabel{#2}%
   \label{#1}\endgroup
}
\makeatother
\makeatletter
\newcommand{\rmh}[1]{\mathpalette{\raisem@th{#1}}}
\newcommand{\raisem@th}[3]{\hspace*{-1pt}\raisebox{#1}{$#2#3$}}
\makeatother

\newcommand{\descref}[2]{\hyperref[#1]{\textnormal{\textcolor{black}{(}\textcolor{blue}{\bf #2}\textcolor{black}{)}}}}

\newcommand{\dref}[2]{\hyperref[#1]{\textcolor{black}{(}\textcolor{blue}{\bf #2}\textcolor{black}{)}}}
\newcommand{\be} {\begin{eqnarray}}
\newcommand{\ee} {\end{eqnarray}}
\newcommand{\Bea} {\begin{eqnarray*}}
\newcommand{\Eea} {\end{eqnarray*}}

\newcommand{\de} {\delta}

\newcommand{\La} {\Lambda}

\newcommand{\f}{\infty}

\newcommand{\R}{\mathbb{R}}

\newcommand{\sgn}{\mathop\mathrm{sgn}}

\newcommand{\D}{\Delta}
\def\d{\delta}
\newcommand{\umax}{\overline{u}}
\newcommand{\umin}{\underline{u}}
\newcommand{\Z}{\mathbb{Z}}
\newcommand{\sumj}{\sum_{j \in \Z}}
\newcommand{\BV}{\textrm{BV}}
\newcommand{\mM}{\mathcal{M}}
\newcommand{\jph}{j+1/2}
\newcommand{\jmh}{j-1/2}
\newcommand{\mF}{\mathcal{F}}
\newcommand{\co}{\mathrm{co}}
\newcommand{\barf}{\bar{f}}

\DeclareMathOperator{\TV}{TV}

\newcommand{\norm}[1]{\left|\hspace{-0.2mm}\left| #1 \right|\hspace{-0.2mm}\right|}
\newcommand{\abs}[1]{\left| #1\right|}

\newcommand{\esslim}{\operatorname*{ess\,lim}}

\newcommand{\M}{\mathcal{M}}

\newcommand{\Linf}{L^{\infty}}

\newcommand{\Do}{\Pi_T}

\newcommand{\dx}{\, dx}

\newcommand{\Dt}{\Delta t}
\newcommand{\Dx}{\Delta x}

\newcommand{\mC}{\mathcal{C}}

\newcommand{\mI}{\mathcal{I}}
\newcommand{\mJ}{\mathcal{J}}
\newcommand{\mK}{\mathcal{K}}

\newcommand{\ud}{u^{\d}}
\newcommand{\Pd}{P^{\d}}
\newcommand{\su}{\mathsf{u}}

\def\d{\delta}
\def\D{\Delta}

\def\norm#1{\left\|#1\right\|}

\def \f12{{\frac12}}

\def \jph{j + \smash{\f12}}
\def \jp3h{j + \smash{\frac32}}
\def \jph{j + \smash{\f12}}
\def \jmh{j - \smash{\f12}}

\newcounter{whitney}
\refstepcounter{whitney}

\newcounter{ineqcounter}
\refstepcounter{ineqcounter}
\makeatletter
\def\ps@pprintTitle{%
\let\@oddhead\@empty
\let\@evenhead\@empty
\def\@oddfoot{}%
\let\@evenfoot\@oddfoot}
\makeatother
\usepackage[doublespacing]{setspace}
\usepackage[titletoc,toc,page]{appendix}

\usepackage{pgf,tikz}
\usetikzlibrary{arrows}
\usetikzlibrary{decorations.pathreplacing}
\begin{document}

\begin{frontmatter}

 \title{A convergence rate result for front tracking 
 approximations of conservation laws with discontinuous flux}

 	\author[myaddress1]{Shyam Sundar Ghoshal}
 	\ead{ghoshal@tifrbng.res.in}


 	\address[myaddress1]{Centre for Applicable Mathematics,Tata Institute of Fundamental Research, Post Bag No 6503, Sharadanagar, Bangalore - 560065, India.}

 	\author[myaddress2]{John D. Towers}
 	\ead{john.towers@cox.net}

 	\address[myaddress2]
 	{MiraCosta College, 3333 Manchester Avenue, Cardiff-by-the-Sea, CA 92007-1516, USA.}

\begin{abstract}
	We consider the initial value problem for a scalar conservation law in one space dimension with a
	single spatial flux discontinuity, the so-called two-flux problem. We prove that a well-known 
	front tracking algorithm
	has a convergence rate of at least one-half. The fluxes are required to be smooth, but are not required to be
	convex or concave, monotone, or even unimodal. We require that there are no more than
	finitely many flux crossings, but we do not require that they satisfy the so-called crossing condition.
	If both fluxes are strictly increasing or strictly decreasing then our analysis yields a
	convergence rate of one, in agreement with a recent result. Similarly, if the fluxes
	are equal, i.e., there is no flux discontinuity, we obtain a convergence rate of one
	in this case also,
	in agreement with a classical result. 
	The novelty of this paper is that the class of discontinuous-flux conservation laws
	for which there is a front tracking error estimate is expanded, and that the method of analysis is new;
	we do not use the Kuznetsov lemma which is commonly used for this type of analysis.  		
	\end{abstract}
	
\end{frontmatter}

\section{\bf Introduction}\label{sec_intro}

The subject of this paper is approximation via front tracking to solutions 
of an initial value problem for a scalar conservation law having a spatially discontinuous flux:
\begin{equation}\label{u_cauchy}
	\textrm{Problem $P$:} \quad
	\left\{
	\begin{split} 
		&u_t+\mF(x,u)_x   = 0 \quad 
		\textrm{for $(x,t)\in \Pi_T:=\R \times (0,T)$,} \\
		&\mF(x,u):= H(-x)g(u) + H(x)f(u),\\
		&u(x,0) =u_0(x) \quad  \textrm{for $x\in \R$.}  
	\end{split} 
	\right .
\end{equation}
Here $H(x)$ is the Heaviside function. Thus, the flux $\mathcal{F}(x,u)$ of this 
conservation law has a spatial dependence that is discontinuous at $x=0$  
if the functions $f$ and $g$ are different.

The front tracking algorithm \cite{Holden_Risebro,risebro_intro} computes approximate solutions $u^{\d}$ of Problem P by 
producing the exact solution to
a simpler problem $P^{\d}$ where  the flux $\mF$ is replaced by a piecewise linear approximation
$\mF^{\d}$ and the initial data $u_0$ is replaced by a piecewise constant approximation $u_0^{\d}$:
\begin{equation}\label{u_cauchy_p}
	\textrm{Problem $\Pd$:} \quad
	\left\{
	\begin{split} 
		&\ud_t+\mF^{\d}(x,\ud)_x   = 0 \quad 
		\textrm{for $(x,t)\in \Pi_T$,} \\
		&\mF^{\d}(x,u):= H(-x)g^{\d}(u) + H(x)f^{\d}(u),\\
		&\ud(x,0) =\ud_0(x) \quad  \textrm{for $x\in \R$.}  
	\end{split} 
	\right .
\end{equation}
Here $\d$ is a parameter that defines the level of discretization. Our main result (Theorem~\ref{theorem_rate}) is that the
front tracking algorithm described below produces approximate solutions that
converge to the so-called vanishing viscosity solution (Definition~\ref{weak_def}) at a rate of at least $O(\d^{1/2})$.
If $f$ and $g$ are both strictly monotone, we obtain 
 a convergence rate of $O(\d)$ (Theorem~\ref{theorem_rate_inc}), in agreement 
with a recent result of Ruf \cite{ruf_ft}. Finally, 
we also obtain 
a convergence rate of $O(\d)$
if $f=g$ (Theorem~\ref{theorem_rate_f_eq_g}), in agreement with a classical result due to
Lucier \cite{lucier_moving_mesh}. Our method of analysis is different from the
methods found in
\cite{ruf_ft} and \cite{lucier_moving_mesh}; for example we do not invoke
Kuznetsov's lemma \cite[Theorem 3.11]{Holden_Risebro}.

Front tracking is both a theoretical tool and a practical computational algorithm. It is more
complex to implement than finite difference / finite volume methods, but it is computationally
rapid and devoid of the numerical dissipation produced by finite difference / finite volume methods.

Before specifying the data of the problem we will need the following definition:
\begin{definition}[flux crossing]
We call $u_{\chi} $ 
a flux crossing if $f(u_\chi)=g(u_\chi)$, and for some $\epsilon>0$ it holds that
\begin{equation*}
	\textrm{$\left(g(u_L)-f(u_L)\right)
	\left(g(u_R)-f(u_R)\right)<0$ for all 
	$u_L \in (u_{\chi}-\epsilon,u_{\chi})$,
	$u_R \in (u_{\chi},u_{\chi}+ \epsilon)$}.
\end{equation*}
In less formal terms, a flux crossing $u_{\chi}$ is the location of a zero of the function
$f(u)-g(u)$ where the sign of $f-g$ changes.
\end{definition}

\begin{assumption}\label{assumptions_data}
We make the following assumptions about the data. In what follows
$\umin, \umax \in \R$ with $\umin<\umax$. 

\begin{Assumptions}

\item \label{fg_endpoints} $g(\umin) = f(\umin)$, $g(\umax) = f(\umax)$,

\item \label{fg_C2} $g, f \in C^2([\umin,\umax])$,

\item \label{nonlin_deg}
$f$ and $g$ are not linear on any nondegenerate interval,

\item \label{finite_crossings}
$f(u)\neq g(u)$ for $u \in (\umin,\umax)$, except for finitely many (possibly zero) flux crossings in the interval $(\umin,\umax)$,

\item \label{u0_Linf_bv} $u_0 \in \Linf(\R) \cap \BV(\R)$, and $u_0(x) \in [\umin,\umax]$ for all $x \in \R$,

\item \label{init_data_large_x} for some $X>0$, $u_0(x) = u_L$ for $x\le -X$ 
and $u_0(x) = u_R$ for $x \ge X$.

\end{Assumptions}
\end{assumption}

\begin{remark}\normalfont
Assumption~\ref{finite_crossings} requires that there be only finitely
many flux crossings, but we do not require that any of them satisfy
the so-called ``crossing condition'' \cite{KRT:L1,KT:VV}.
\end{remark}

\begin{remark}\normalfont
Items \ref{nonlin_deg} and \ref{finite_crossings} are required so that we can use
the results of \cite{KT:VV} concerning the convergence of a certain Godunov scheme
(which we use as an analytical tool) to a vanishing viscosity solution. 
In the special case where
both $f$ and $g$ are strictly increasing or strictly decreasing, the vanishing viscosity theory
simplifies greatly, and those assumptions are not needed. Similarly, if $f=g$ we do not need
\ref{nonlin_deg} or \ref{finite_crossings}. 
\end{remark}

{\bf The front tracking algorithm.} To describe the front track algorithm we must specify the construction
of $f^{\d}$, $g^{\d}$, and $u_0^{\d}$ for fixed $\d>0$. 
To specify  $f^{\d}$ and $g^{\d}$, choose a positive integer $K$, and 
breakpoints $\su_0, \su_1, \ldots, \su_K$ such that
 \begin{equation}
\umin = \su_0 < \su_1 < \cdots < \su_K=\umax, \quad 
\textrm{$\su_{k+1}-\su_k \le \d$ for $k=0,\ldots,K-1$.}
 \end{equation}
 Then $f^{\d}$ and $g^{\d}$ are defined to be the piecewise linear interpolants to $f$ and $g$ 
 such that $f^{\d}(\su_k) = f(\su_k)$ and $g^{\d}(\su_k) = g(\su_k)$ for $k=0,\ldots,K$.
  
The front tracking initial data $u_0^{\d}$ is constructed satisfying the following conditions:
\begin{equation}\label{v_init}
\begin{split}
&\textrm{1. $u_0^{\d}$ is piecewise constant with
finitely many jumps, and such that}\\
& u_0^{\d}(x) \in [\umin,\umax]\,\, \forall x\in \R, \quad \TV(u_0^{\d}) \le \TV(u_0), \\
&\textrm{2. condition~\ref{init_data_large_x} is satisfied with the same $X$ as for $u_0$, and}\\
&\textrm{$\norm{u_0^{\d}-u_0}_{L^1(\R)} = O(\d)$ as $\d \rightarrow 0$.}
\end{split}
\end{equation}

\begin{remark}\label{remark_delta}\normalfont
$f^{\d}$ and $g^{\d}$ clearly satisfy Assumption~\ref{fg_endpoints}, and for $\d$ sufficiently small,
Assumption~\ref{finite_crossings} also holds for $f^{\d}$, $g^{\d}$.
 It follows from Assumption~\ref{fg_C2} that $f, g \in \textrm{Lip}([\umin,\umax])$. Let $L_f$ and $L_g$ denote Lipschitz constants
for $f$ and $g$. It is clear that also $f^{\d}, g^{\d} \in \textrm{Lip}([\umin,\umax])$, and 
\begin{equation}\label{lip_delta}
L_{f^{\d}} \le L_f, \quad  L_{g^{\d}} \le L_g.
\end{equation}
 Moreover, it is clear from \eqref{v_init} that
$u_0^{\d}$ satisfies Assumptions \ref{u0_Linf_bv} and \ref{init_data_large_x}. 
Thus all parts of Assumption~\ref{assumptions_data} are satisfied for the data
associated with the front tracking algorithm, with exception of Assumption~\ref{nonlin_deg}. This assumption
is present to guarantee the existence of one sided traces along the interface at $x=0$.
With the addition of Assumption~\ref{finite_fronts}
below the required traces of $u^{\d}$ exist without Assumption~\ref{nonlin_deg}.
\end{remark}
  
 In view of Remark~\ref{remark_delta}, with $\d>0$ fixed and $f^{\d}$, $g^{\d}$, and $u_0^{\d}$ in hand, the front tracking algorithm then computes the exact vanishing viscosity solution, defined below (Definition~\ref{weak_def}), of problem $P^{\d}$. 
  The solution of problem $P^{\d}$ is piecewise constant, with discontinuities (fronts), some of which are moving. 
 The algorithm keeps track of the location of each front as time advances.
 Each time two fronts collide, the piecewise constant solution is updated by solving a Riemann problem.
 Each such Riemann problem is constructed so that the $\Gamma$ condition (Definition~\ref{def:gamma_condition})
 is satisfied, and the resulting front tracking solution is thus a vanishing viscosity solution (Definition~\ref{weak_def}).
 
 In order for the front tracking algorithm to be well-defined, there cannot be more than finitely many
 fronts present in $u^{\d}$ at any finite time $t \in [0,T]$. That this is the case has been proven in
some important cases \cite{BaitiJenssen,Coclite_Risebro,GNPT:2007,KlingbroI,piccoli_tournus,risebro_intro}, 
but not in all cases allowed by our assumptions about $f$ and $g$. Thus
 we add the following assumption:
\begin{assumption}\label{finite_fronts}
For any $t \in [0,T]$, $u^{\d}(\cdot,t)$ is piecewise constant with finitely many discontinuities.
\end{assumption}

The subject of conservation laws with discontinuous flux has been an active research area  during
the last several decades.
See e.g.,   
\cite{ADGG:2011, AGJ:2005, And_Cances, AKR:L1, AudussePerthame,badwaik_ruf,badwaik_random,BaitiJenssen, bressan_regulated, bkrt2, BKT:EO, crasta_structure, Diehl:2009, GNPT:2007,GJT_2019,GTV_2020,gtv_panov, KRT:L1, KT:LxF, KT:VV,KlingbroI, 
shearer, mishra_handbook, Panov2009a, piccoli_tournus, risebro_intro, seguin01, shen_polymer, towers1}
for a partial list of references.
These equations arise in a number of areas of application including
vehicle traffic flow in the presence of abruptly varying road conditions \cite{BGKT,GNPT:2007},
sedimentation in settling tanks \cite{bkrt2, Diehl:1995},
polymer flooding in oil recovery \cite{shen_polymer},
two phase flow in porous media \cite{AdimurthiJaffreGowda,And_Cances,GR:1992,Kaasschieter},
and particle size segregation in granular flow \cite{shearer}.

{\bf Vanishing viscosity solutions.}
Even in the classical setting where $f=g$, and even if the initial data $u_0$ is smooth,
solutions develop discontinuities, and so we seek weak solutions. However, weak
solutions are generally not unique. Because of this an additional condition, referred to as an
entropy condition, or admissibility condition, is  required.  In the case where $f=g$, the
Kru\v{z}kov entropy solution \cite{Holden_Risebro} is generally sought, and this resolves the uniqueness issue. 
In the case of a spatially discontinuous flux, the Kru\v{z}kov entropy inequalities
do not make sense, and so some other entropy condition is required. The situation is further
complicated by the fact that
there is not just a single notion of entropy solution for the problem \eqref{u_cauchy}, see \cite{AKR:L1}. We will
work with the so-called vanishing viscosity solution \cite{AKR:L1, AKR:VV, And_Mit,bressan_regulated,KT:VV}, which we now set out to define.

The notion of traces appears in the definition of vanishing viscosity
solution, by which we mean the following limits:
\begin{equation}\label{traces}
	u(0^-,t) := \esslim_{x\uparrow 0} u(x,t),
	\qquad u(0^+,t) := \esslim_{x \downarrow 0} u(x,t).
\end{equation}
In the sequel we will sometimes use the abbreviations 
$u(0^-,t) = u_-$, $u(0^+,t) = u_+$ for the traces 
appearing in \eqref{traces}.

Let $\co(a,b)$ denote 
the interval $[\min(a,b),\max(a,b)]$. The following is Diehl's $\Gamma$ 
condition \cite{AKR:L1,AKR:VV, Diehl:1995, Diehl:2009}, which will be used below in the definition of vanishing viscosity solutions:
\begin{definition}[$\Gamma$ condition]\label{def:gamma_condition}
The pair $(u_{-},u_{+})$ satisfies the $\Gamma$ condition if
\begin{equation*}
	\begin{split}
	&\textrm{$g(u_{-}) = f(u_{+})$, and there exists $u_{\Gamma} \in \co(u_{-},u_{+})$  such that}\\
	&\textrm{$(u_{+} - u_{\Gamma})(f(z)-f(u_{+})) \ge 0 \,\, \forall z \in \co(u_{+},u_{\Gamma})$ and}\\
	&\textrm{$(u_{\Gamma}- u_{-})(g(z)-g(u_{-})) \ge 0 \,\, \forall z \in \co(u_{-},u_{\Gamma})$.}
	\end{split}
\end{equation*}
\end{definition}

\begin{remark}\normalfont
The condition $g(u_{-}) = f(u_{+})$ appearing in the definition 
is the familiar Rankine-Hugoniot condition.
\end{remark}

\begin{definition}[Vanishing viscosity solution, \cite{AKR:L1,AKR:VV}]\label{weak_def}
 A measurable function $u: \Pi_T \rightarrow \R$ is a 
{\em vanishing viscosity solution}  of the initial value problem \eqref{u_cauchy} if it 
satisfies the following conditions:
\begin{Definitions}
	\item $u\in L^{\infty}(\Pi_T)$; $u(x,t)\in [\underline{u},\overline{u}]$ for 
	a.e.~$(x,t)\in \Pi_T$. \label{def:weak1}

	\item For all test functions $\phi\in \mathcal{D}(\mathbb{R} \times [0,T))$,
	\begin{align}
		\label{weak_cond}
		\begin{split}
			\iint_{\Pi_T}\bigl(u \phi_t &
			+ \mathcal{F}(x,u) \phi_x \bigr)\,dx \, dt
			+\int_{\mathbb{R}} u_0(x) \phi(x,0) \, dx =0.
		\end{split}
        \end{align}\label{def:weak2}

	\item For any test function $0\le \phi \in \mathcal{D}(\mathbb{R} \times [0,T))$ 
	which vanishes for $x \ge 0$,
	\begin{align}\label{ent_g}
                \begin{split} 
			& \iint_{\Pi_T}\Bigl(|u-c|\phi_t  +  \sgn(u-c)\bigl(g(u)-g(c)\bigr)\phi_x
			\Bigr)\,dx\,dt \\ & 
			\qquad
			+ \int_{\mathbb{R}} |u_0-c| \phi(x,0) \,dx \geq 0,
		 	\quad \forall c \in [\underline{u},\overline{u}], 
		 \end{split} 
	\end{align}
	and for any test function $0\le \phi \in \mathcal{D}(\mathbb{R} \times [0,T))$ 
	which vanishes for $x \le 0$,
	\begin{align}\label{ent_f}
                \begin{split} 
			& \iint_{\Pi_T}\Bigl(|u-c|\phi_t  
			+\sgn(u-c)\bigl(f(u)-f(c)\bigr)\phi_x \Bigr)\,dx\,dt\\
			& \qquad 
			+  \int_{\mathbb{R}} |u_0-c| \phi(x,0) \,dx  \geq 0, 
			\quad \forall c \in [\underline{u},\overline{u}].
                \end{split} 
        \end{align}\label{def:weak3}

        \item \label{ent_gamma} For a.e. $t \in (0,T)$, the traces $u_-:=u(0^-,t)$ 
	and $u_+:=u(0^+,t)$  satisfy the $\Gamma$ condition. 
	\label{def:weak4}
\end{Definitions}

A measurable function $u:\Pi_T \rightarrow \mathbb{R}$ satisfying conditions 
\ref{def:weak1} and \ref{def:weak2} is called a weak solution of the 
initial value problem \eqref{u_cauchy}.
\end{definition}

\begin{remark}\normalfont
Note that condition~\ref{ent_gamma} assumes the existence of traces.
For the solution $u$ of problem $P$, 
existence of traces follows from Assumption~\ref{nonlin_deg} and condition~\ref{def:weak3},
see Lemma 3.1 of \cite{bkrt2}. For the front tracking solution $u^{\d}$ of problem $P^{\d}$, existence of traces follows from
Assumption~\ref{finite_fronts}.
\end{remark}

\begin{theorem}[Uniqueness of vanishing viscosity solutions, \cite{AKR:L1,AKR:VV}]\label{thm:unique}
Let $u$, $v$ be two vanishing viscosity solutions 
in the sense of Definition \ref{weak_def} of 
the initial value problem \eqref{u_cauchy}, with initial 
data $u_0,v_0\in L^\infty(\mathbb{R};[0,1])$, $\abs{u_0-v_0} \in L^1(\R)$. 
Then, for a.e.~$t\in (0,T)$,
\begin{equation*}
	\int_{\R} \bigl| u(x,t)-v(x,t) \bigr|\, dx \le 
	\int_{\R} \bigl|u_0(x)-v_0(x) \bigr|  \, dx.
\end{equation*}
In particular, there exists at most one vanishing viscosity solution 
of \eqref{u_cauchy}. If Assumption~\ref{finite_fronts} is satisfied the same result holds for vanishing viscosity solutions of
the front tracking problem \eqref{u_cauchy_p}.
\end{theorem}

The vanishing viscosity
solution has other characterizations, such as the minimal jump condition
of Gimse and Risebro \cite{GR:1991, GR:1992}  and
the $\Gamma_1$-condition of \cite{KT:VV}.
When the crossing condition \cite{KRT:L1,KT:VV} is satisfied, 
the entropy solution of \cite{KRT:L1} is also the vanishing viscosity solution. 
Under certain conditions (eg., as in Example~\ref{example_1}) the 
vanishing viscosity solution coincides with the optimal entropy solution of
\cite{AGJ:2005}.
Reference \cite{guerra_shen_bkwd} uses semigroup theory to define
vanishing viscosity solutions.
The vanishing viscosity solution concept has been generalized to multidimensional
conservation laws with discontinuous flux \cite{AKR:VV,And_Mit,bulicek2,Crasta:2015aa,Panov2009a}, and
to the case where the flux has discontinuities that are both spatial and temporal \cite{bressan_regulated}.

\begin{example}\label{example_1} \normalfont
A simple but important example 
which has been studied extensively
has the form
\begin{equation}\label{example_traffic}
u_t + (k(x) p(u))_x = 0, \quad p(u) = u(1-u), 
\quad \textrm{$\umin=0$, $\umax=1$},
\end{equation}
with $k(x) = (1-H(x))k_L + H(x)k_R$ for some $k_L, k_R >0$.
Here $g(u)=k_L p(u)$, $f(u)=k_R p(u)$.
Our front tracking error estimate applies
to this example, and to our knowledge it is the first such.
The conservation law  \eqref{example_traffic} can be viewed as a model of vehicle traffic on a unidirectional road (the so-called
Lighthill-Witham-Richards model), where the parameter $k(x)$ is a spatially dependent
maximum vehicle speed.
In this case there is no flux crossing, so Assumption~\ref{finite_crossings} is
trivially satisfied. For this example the vanishing viscosity solution is the
same as many of the other types of solution that have been defined, including the minimal jump solution,
the entropy solution of \cite{KRT:L1}, and the optimal entropy solution. 
Although the spatial variation of the solution is not generally bounded, one can obtain
a bound on the spatial variation of the so-called singular mapping \cite{AdimurthiJaffreGowda,KlingbroI,seguin01,towers1}. 
Since $dp/du$ only vanishes at a single point, the singular mapping is invertible, which guarantees
the existence of traces for $u$ \cite{KRT:L1}. 
If the break points 
$\{\su_k\}$ are
chosen so that $dp^{\d}/du \neq 0$ on any nontrivial interval, then one also has
traces for the front tracking solution $u^{\d}$,
independent of Assumption~\ref{finite_fronts}. 
\end{example}

When $f=g$, Lucier \cite{lucier_moving_mesh} proved a first order convergence rate for front tracking. 
More recently reference \cite{solem_ft} has proven higher convergence rates in the Wasserstein norms,
for the case $f=g$.
There are currently only a few convergence rate results for algorithms
pertaining to nonlinear conservation laws with discontinuous flux, to our knowledge only \cite{badwaik_ruf},
\cite{gtv_panov} and \cite{ruf_ft}. Reference \cite{badwaik_ruf}
proves a convergence rate of one half for a monotone finite volume scheme in
the case where the flux has finitely many spatial discontinuities and
the dependence of the flux upon $u$ is strictly monotone.
Under similar assumptions reference~\cite{ruf_ft} proves a first order
convergence rate for a front tracking algorithm. 
Reference \cite{gtv_panov} proves a convergence rate of one half for a Godunov-type
finite volume algorithm applied to conservation law whose flux is of Panov-type and  has spatial discontinuities.
References \cite{badwaik_ruf},
\cite{gtv_panov} and \cite{ruf_ft} employ Kuznetsov-type lemmas \cite{Holden_Risebro,kuznetsov}.
The results of the present paper do not use a Kuznetsov-type lemma. In order to
use a Kuznetsov-type lemma a spatial $\BV$ bound is required for the solution
and the approximations. $\BV$ bounds are available when both $f$ and $g$ are
strictly monotone, but in general no $\BV$ bound is available when $f \neq g$ \cite{ADGG:2011,SSG1,SSG2}.
Thus the novelty of the present paper is twofold. The class of discontinuous-flux conservation laws
for which there is a front tracking error estimate is expanded, and the method of analysis is new.

In Section~\ref{Sec:scheme} we describe a Godunov-type finite volume scheme that is used
as an analytical tool. Section~\ref{sec_rate_conv} contains the statement and proof of Theorem~\ref{theorem_rate}, our main result, 
which is a
rate of convergence of $O(\d^{1/2})$ for the front tracking algorithm. In Section~\ref{sec_conv_bv} we apply
our result to the simpler situation where a spatial $\BV$ bound is known for the solution. We
arrive at a rate of $O(\d)$ in this situation. This includes the case where $f=g$, and also the case where $f$ and $g$ are both strictly monotone. The $O(\d)$ result correlates
with previously established results in these two cases \cite{lucier_moving_mesh,ruf_ft}. 

\section{A Godunov finite volume algorithm} \label{Sec:scheme}
Although the convergence rate result of this paper is for a front tracking algorithm, a certain finite volume algorithm makes an
appearance as an analytical tool. 
Specifically, we use the Godunov-type scheme of \cite{Diehl:1995,KT:VV}, which we describe next.

We begin by
discretizing the spatial domain $\R$ into cells
$$
I_j= \bigl(x_{j}-\Dx/2,x_{j}+\Dx/2\bigr]
=\bigl(x_{j-\frac12},x_{j+\frac12}\bigr],
$$
where $x_j=j\Dx$ for $j\in \mathbf{Z}$.  Similarly, the time
interval $[0,T]$ is discretized via $t^n=n\Dt$ for $n=0,\dots,N$,
where the integer $N$ is such that $N\Dt\in [T,T+\D t)$, resulting 
in the time strips
$$
I^n=[t^n,t^{n+1}).
$$
Here $\Dx>0$ and $\Dt>0$ denote the spatial and temporal
discretization parameters respectively.
With $\lambda = \Dt / \Dx$, we will assume that the 
mesh size $\Delta :=(\Delta x, \Delta t)$ approaches
zero with $\lambda$ fixed.
We will assume that the following CFL condition holds:
\begin{equation}\label{CFL}
	\lambda \cdot \max \left( L_f , L_g\right) \le 1/2.
\end{equation}

We will use $U_j^n$ to denote the resulting finite
difference approximation; $U_j^n \approx u(x_j,t^n)$.
The initial data is discretized via cell averages:
\begin{equation}\label{disc_u}
     U_j^0 = \frac{1}{\Dx} \int_{x_{\jmh}}^{x_{\jph}} u_0(x)\dx.
\end{equation}
Let $\chi_j(x)$ and $\chi^n(t)$ denote the characteristic functions for
the intervals $I_j$ and $I^n$, respectively.  
The finite difference solution $\left\{U_j^n\right\}$
is extended to all of $\Do$ by defining
\begin{equation*}
     u^{\D}(x,t) = \sum_{n= 0}^N \sum_{j \in \mathbf{Z}}\chi_j(x) \chi^n(t)U_j^n,
     \qquad (x,t)\in \Do.
     \label{eq:uDdef}
\end{equation*}
We use $\D_+$ and $\D_-$ to designate
the difference operators in the $x$ direction, e.g.,
$$
\D_+ Z_j = Z_{j+1} - Z_j, \quad
\D_- Z_j = Z_{j} - Z_{j-1}.
$$
For a grid function $\{Z_j^n\}$ and $q \in C([\umin,\umax])$ we use
the notation
\begin{equation}
\norm{Z^n}_{\infty} = \max_{j \in \Z} \abs{Z_j^n},
\quad \norm{q}_{\infty} = \max_{u\in [\umin,\umax]} \abs{q}.
\end{equation}

\noindent
Given a Lipschitz function $q(u)$, we denote by 
$\bar{q}(v,u)$ the associated Godunov numerical flux function:
\begin{equation}\label{gdnvflux_0}
	\bar{q}(v,u)=
	\begin{cases}
		\min\limits_{w \in [u,v]} q(w), &\textrm{if}
		\hspace{3pt}u \leq v, \\
		\max\limits_{w \in [v,u]} q(w), &\textrm{if}
		\hspace{3pt}u > v.
	\end{cases}
\end{equation}
The function $\bar{q}(v,u)$ is a monotone numerical flux \cite{CranMaj:Monoton,Holden_Risebro}, i.e., nondecreasing with
respect to $u$, nonincreasing with respect to $v$, and satisfies $\bar{q}(u,u) = q(u)$.
It is readily verified that if $L_q$ is a Lipschitz constant for $q$, then
\begin{equation}\label{barq_lip}
\abs{\bar{q}(\tilde{u},v) - \bar{q}(u,v)} \le L_q \abs{\tilde{u}-u}, \quad
\abs{\bar{q}(u,\tilde{v}) - \bar{q}(u,v)} \le L_q \abs{\tilde{v}-v}
\end{equation}
if $u,v,\tilde{u},\tilde{v} \in [\umin,\umax]$.
The approximate solution is advanced 
from one time level to the next 
via the following finite difference formula:
\begin{equation}\label{scheme}
	U_j^{n+1} = U_j^n - \lambda \D_-h_{\jph}(U_{j+1}^n,U_j^n),
\end{equation}
where $h_{\jph}$ is the spatially dependent Godunov numerical flux:
\begin{equation}\label{gdnvflux}
	h_{\jph}(v,u)=
	\begin{cases}
		\bar{g}(v,u), &\textrm{if}
		\hspace{3pt} x_{\jph}<0, \\
		\bar{f}(v,u), &\textrm{if}
		\hspace{3pt} x_{\jph}>0.
	\end{cases}
\end{equation}
Here $\bar{f}, \bar{g}$ refer to the Godunov fluxes consistent with $f, g$ defined by
\eqref{gdnvflux_0}.

 Similarly, $U^{\d,n}_j$ denotes the numerical approximation when the Godunov scheme above is applied 
 with flux $\mF^{\d}$ and initial data $u^{\d}_0$. The numerical flux $h_{\jph}^{\d}$ is the version of \eqref{gdnvflux} that
 results when
 $\bar{f}, \bar{g}$ are replaced by $\bar{f}^{\d}, \bar{g}^{\d}$.  The CFL condition \eqref{CFL} applies equally
 well to the scheme for $U_j^{\d,n}$, due to \eqref{lip_delta}.
 
 Recall that according to Assumption~\ref{init_data_large_x}, $u_0(x) = u_L$ for $x \le -X$.
 Referring to the Godunov algorithm used to define $U_j^n$  and $\hat{U}_j^{\d,n}$, it is
 clear that $U_j^n = u_L$  and $\hat{U}_j^{\d,n}=u_L$ for $j<0$ with $\abs{j}$ sufficiently large.
 Associated with $U_j^n$  and $\hat{U}_j^{\d,n}$ are the following discrete indefinite integrals:
 \begin{equation}
 \hat{U}_j^n = \D x \sum_{i\le j} \left(U_i^n - u_L \right),
 \quad
 \hat{U}_j^{\d,n} = \D x \sum_{i\le j} \left(U_i^{\d,n} - u_L \right).
 \end{equation}
 Note that
 \begin{equation}\label{U_Uhat_formula}
 \D_- \hat{U}_j^n/\D x = U_j^n - u_L, \quad  \D_- \hat{U}_j^{\d,n}/\D x = U_j^{\d,n} - u_L.
 \end{equation}
By applying the same discrete integration operation to the marching formula \eqref{scheme} we obtain
 \begin{equation}\label{HJ_scheme}
 \hat{U}_j^{n+1} =  \hat{U}_j^{n} - \D t\, h_{\jph}(U_{j+1}^n, U_j^n),
 \quad
 \hat{U}_j^{\d,n+1} =  \hat{U}_j^{\d,n} - \D t\, h^{\d}_{\jph}(U_{j+1}^{\d,n}, U_j^{\d,n}).
 \end{equation}
 
  \noindent
  These grid functions are then extended to functions defined on $\Do$ via 
 \begin{equation}\label{extend_hats}
 \hat{u}^{\D}(x,t) 
 = \sum_{n=0}^N \sumj \chi^n(t) \chi_j(x) 
 \left(
 \hat{U}_{j-1}^n + \frac{x-x_{\jmh}}{\Dx}\left(\hat{U}_{j}^n - \hat{U}_{j-1}^n\right)
  \right),
 \end{equation}
 and similarly for $\hat{u}^{\d,\D}(x,t)$.
 From \eqref{extend_hats} one finds that
 \begin{equation}\label{hat_disc_integrals}
  \hat{u}^{\D}(x,t)  =  \int_{-\infty}^x \left(u^{\D}(y,t)-u_L\right) \,dy, \quad
  \hat{u}^{\d,\D}(x,t)  =  \int_{-\infty}^x \left(u^{\d,\D}(y,t)-u_L\right) \,dy.
 \end{equation}

  \begin{theorem}[Theorem 5.4 of \cite{KT:VV}]\label{theorem_convergence}
 Problem $P$ has a unique vanishing viscosity solution, $u$. Similarly,
 Problem $P^{\d}$ has a unique vanishing viscosity solution, $u^{\d}$.
 Let $u^\D$ denote the finite volume approximation to $u$,
and $u^{\d,\D}$ denote the finite volume approximation to $u^{\d}$.
 Then, as $\D \rightarrow 0$, 
 $u^\D \rightarrow u$ and  $u^{\d,\D} \rightarrow u^{\d}$
  in $L^1_{\mathrm{loc}}(\Pi_T)$ and a.e.~in $\Pi_T$.
 \end{theorem}
 
 \begin{remark}\normalfont
 We will make use of the fact, which follows from the analysis in \cite{KT:VV}, that $[\umin,\umax]$
 is an invariant region for all of 
 $u^{\D}$, $u^{\d,\D}$, $u$, $u^{\d}$.
 \end{remark}
  
Associated with the vanishing viscosity solution $u$, we define the indefinite integrals
\begin{equation}\label{hat_def_integrals}
\hat{u}_0(x) = \int_{-\infty}^x \left(u_0(y)-u_L \right)\, dy, \quad \hat{u}(x,t) = \int_{-\infty}^x \left(u(y,t)-u_L\right) \, dy,
\end{equation}
and similarly for $\hat{u}^{\d}_0(x)$ and $\hat{u}^{\d}(x,t)$.

  \begin{remark}\normalfont \label{remark_conv}
  By using Lemma 5.1 of \cite{KT:VV}, which contains a discrete time continuity estimate, we also
  have $u, u^{\d} \in C([0,T]:L^1_{\mathrm{loc}}(\R))$. Moreover, 
 using the argument in the proof of Theorem 1 of \cite{CranMaj:Monoton},
  for $t \in [0,T]$, we have
  \begin{equation}\label{L1_conv_time_slice}
  \norm{u(\cdot,t) - u^{\D}(\cdot,t)}_{L^1(\R)} \rightarrow 0, \quad
   \norm{u^{\d}(\cdot,t) - u^{\d,\D}(\cdot,t)}_{L^1(\R)} \rightarrow 0 \,\,
   \textrm{as $\D \rightarrow 0$}.
  \end{equation}
Using \eqref{hat_disc_integrals}, \eqref{hat_def_integrals} and \eqref{L1_conv_time_slice}  it is readily established that for $t \in [0,T]$
\begin{equation}\label{u_hat_conv}
\norm{\hat{u}(\cdot,t) - \hat{u}^{\D}(\cdot,t)}_{L^{\infty}(\R)} \rightarrow 0, \quad
   \norm{\hat{u}^{\d}(\cdot,t) - \hat{u}^{\d,\D}(\cdot,t)}_{L^{\infty}(\R)} \rightarrow 0 \,\,
   \textrm{as $\D \rightarrow 0$}.
\end{equation}
\end{remark}

\begin{lemma}\label{lemma_gdv_approx}
For $u,v \in [\umin,\umax]$,
\begin{equation}\label{gdv_approx}
\abs{\bar{f}^{\d}(v,u)-\bar{f}(v,u)} \le \norm{f^{\d}-f}_{\infty}, \quad
\abs{\bar{g}^{\d}(v,u)-\bar{g}(v,u)} \le \norm{g^{\d}-g}_{\infty}.
\end{equation}
\end{lemma}

\begin{proof}
It suffices to prove  the first part of \eqref{gdv_approx}. Also we assume that $u\le v$; the 
other case is similar. Recalling \eqref{gdnvflux_0},
\begin{equation}\label{f_gdv_comp_1}
\textrm{$\bar{f}(v,u) = \min_{w \in [u,v]}f(w) = f(w^*)$ for some $w^* \in [u,v]$.}
\end{equation}
Similarly,
\begin{equation}\label{f_gdv_comp_2}
\textrm{$\bar{f}^{\d}(v,u) = \min_{w \in [u,v]}f^{\d}(w) = f^{\d}(w^{\d,*})$ 
for some $w^{\d,*} \in [u,v]$.}
\end{equation}
Since $f^{\d}$ is a piecewise linear interpolant of $f$, we have $f(w^*) \le f^{\d}(w^{\d,*})$, and so
\begin{equation}\label{f_gdv_comp_3}
\abs{f(w^*) - f^{\d}(w^{\d,*})} =  f^{\d}(w^{\d,*}) - f(w^*) \le f^{\d}(w^*) - f(w^*)\le \norm{f^{\d}-f}_{\infty}.
\end{equation}
\end{proof}

\section{Proof  of the main theorem}\label{sec_rate_conv}

What follows is a sequence of lemmas leading up to Theorem~\ref{theorem_rate}, 
our main theorem.

\begin{lemma}\label{lemma_uhat_diff}
For each $t \in [0,T]$, we have
\begin{equation}\label{uhat_diff_1}
\norm{\hat{u}(\cdot,t) - \hat{u}^{\d}(\cdot,t)}_{L^\infty(\R)}
\le \norm{\hat{u}_0- \hat{u}^{\d}_0}_{L^\infty(\R)}
+ t \max \left( \norm{f - f^{\d}}_{\infty}, \norm{g - g^{\d}}_{\infty}\right).
\end{equation}
\end{lemma}

\begin{remark}\normalfont
We will estimate the right side of \eqref{uhat_diff_1} more precisely later, but for now we note that it is finite.
A straightforward calculation yields
\begin{equation}
\norm{\hat{U}^0 - \hat{U}^{\d,0}}_{\infty} \le \norm{u_0- u^{\d}_0}_{L^1(\R)} \le 2X\left(\umax - \umin \right).
\end{equation}
Then recalling \eqref{u_hat_conv}, we have a bound for the first term on the right side of \eqref{uhat_diff_1}.
A bound for the second term follows from the fact that all of $f, g, f^{\delta}, g^{\delta}$ are continuous
on the closed interval $[\umin,\umax]$.
\end{remark}

\begin{proof}
We first prove that
for $n \ge 0$,
\begin{equation}\label{hatU_est_1}
\norm{\hat{U}^n - \hat{U}^{\d,n}}_{\infty}
\le \norm{\hat{U}^0 - \hat{U}^{\d,0}}_{\infty}
+ (n \D t) \max\left(\norm{f-f^{\d}}_{\infty},\norm{g-g^{\d}}_{\infty} \right).
\end{equation}
Then we send $\D \rightarrow 0$ in \eqref{hatU_est_1} and invoke \eqref{u_hat_conv}.

To prove \eqref{hatU_est_1}, it suffices to prove that
\begin{equation}\label{hatU_est_2}
\norm{\hat{U}^{n+1} - \hat{U}^{\d,n+1}}_{\infty}
\le \norm{\hat{U}^n - \hat{U}^{\d,n}}_{\infty}
+ \D t \max\left(\norm{f-f^{\d}}_{\infty},\norm{g-g^{\d}}_{\infty} \right).
\end{equation}
Then \eqref{hatU_est_1} follows by induction on the time
level $n$.

We start on the proof of \eqref{hatU_est_2}.
An application of \eqref{HJ_scheme} yields
\begin{equation}\label{Uhat_diff_1}
\begin{split}
\hat{U}_j^{n+1} - \hat{U}_j^{\d,n+1}
&=  \hat{U}_j^{n} -  \hat{U}_j^{\d,n}
   - \D t \left(h_{\jph}(U_{j+1}^n, U_j^n) - h_{\jph}(U_{j+1}^{\d,n}, U_j^{\d,n})\right)\\
   &\qquad \qquad \quad + \D t \left(h_{\jph}^{\d}(U_{j+1}^{\d,n}, U_j^{\d,n}) - h_{\jph}(U_{j+1}^{\d,n}, U_j^{\d,n})\right)\\
&=  \hat{U}_j^{n} -  \hat{U}_j^{\d,n}
   - \D t \left({{h_{\jph}(U_{j+1}^n, U_j^n) - h_{\jph}(U_{j+1}^n, U_j^{\d,n}) }\over {U_j^n-U_j^{\d,n}}}\right) \left(U_j^n-U_j^{\d,n} \right) \\
   &\qquad \qquad \quad 
   - \D t \left({{h_{\jph}(U_{j+1}^n, U_j^{\d,n}) - h_{\jph}(U_{j+1}^{\d,n}, U_j^{\d,n})} \over {U_{j+1}^n-U_{j+1}^{\d,n}}}\right)\left(U_{j+1}^n-U_{j+1}^{\d,n} \right)\\
 &\qquad \qquad \quad + \D t \left(h_{\jph}^{\d}(U_{j+1}^{\d,n}, U_j^{\d,n}) - h_{\jph}(U_{j+1}^{\d,n}, U_j^{\d,n})\right).   
\end{split}
\end{equation}
We  define 
\begin{equation}
\begin{split}
&\alpha_{\jph}^n 
= \lambda \left({{h_{\jph}(U_{j+1}^n, U_j^n) - h_{\jph}(U_{j+1}^n, U_j^{\d,n}) }\over {U_j^n-U_j^{\d,n}}}\right),\\
&\beta_{\jph}^n 
= -\lambda  \left({{h_{\jph}(U_{j+1}^n, U_j^{\d,n}) - h_{\jph}(U_{j+1}^{\d,n}, U_j^{\d,n})} \over {U_{j+1}^n-U_{j+1}^{\d,n}}}\right),
\end{split}
\end{equation}
and use the identities \eqref{U_Uhat_formula}.
Then \eqref{Uhat_diff_1} becomes
\begin{equation}\label{Uhat_diff_2}
\begin{split}
\hat{U}_j^{n+1} - \hat{U}_j^{\d,n+1}
&= \left(1 - \alpha^n_{\jph} - \beta^n_{\jph}\right) \left(\hat{U}_j^{n} - \hat{U}_j^{\d,n}\right)\\
 &+ \beta^n_{\jph}  \left(\hat{U}_{j+1}^{n} - \hat{U}_{j+1}^{\d,n}\right)
  +  \alpha^n_{\jph}  \left(\hat{U}_{j-1}^{n} - \hat{U}_{j-1}^{\d,n}\right)\\
 &+ \D t \left(h_{\jph}^{\d}(U_{j+1}^{\d,n}, U_j^{\d,n}) - h_{\jph}(U_{j+1}^{\d,n}, U_j^{\d,n})\right).
\end{split}
\end{equation}
Due to the fact that $h_{\jph}$ is a monotone flux, $ \alpha^n_{\jph},  \beta^n_{\jph} \ge 0$. By
invoking the CFL condition \eqref{CFL}, along with \eqref{barq_lip}, we will have $\alpha^n_{\jph} + \beta^n_{\jph} \le 1$.
With these observations, \eqref{Uhat_diff_2} yields
\begin{equation}\label{Uhat_diff_3}
\begin{split}
\abs{\hat{U}_j^{n+1} - \hat{U}_j^{\d,n+1}} 
&\le \max_{i=-1,0,1} \abs{\hat{U}_{j+i}^{n} - \hat{U}_{j+i}^{\d,n}}\\
&+  \D t \abs{h_{\jph}^{\d}(U_{j+1}^{\d,n}, U_j^{\d,n}) - h_{\jph}(U_{j+1}^{\d,n}, U_j^{\d,n})}\\
&\le \norm{\hat{U}^n - \hat{U}^{\d,n}}_{\infty}
+  \D t \abs{h_{\jph}^{\d}(U_{j+1}^{\d,n}, U_j^{\d,n}) - h_{\jph}(U_{j+1}^{\d,n}, U_j^{\d,n})}.\\
\end{split}
\end{equation}
One has $U_j^{\d,n} \in [\umin,\umax]$ (see Lemma 5.1 of \cite{KT:VV}), making it possible to
invoke Lemma~\ref{lemma_gdv_approx} and obtain \eqref{hatU_est_2}.

\end{proof}

We will make use of the observation that there is a $Y^{\D}>0$ such that for all $0\le n \le N$,
\begin{equation}\label{far-field}
\textrm{$U_j^n = u_L$ for $x_j \le -Y^{\D}$, and  $U_j^n = u_R$ for $x_j\ge Y^{\D}$}.
\end{equation}
This also holds for $U^{\d,n}_j$, with the same $Y^{\D}$.
The assertion \eqref{far-field} holds due to  
the fact that the range of influence of the initial interval $[-X,X]$ spreads by no more than
$\D x$ in each direction at each time step. It follows that $Y^{\D}  = X + T/\lambda + O(\D)$.
 By sending $\D \rightarrow 0$, and
taking $1/\lambda = 2 \max(L_f,L_g)$ (as allowed by the CFL condition \eqref{CFL}), we obtain
\begin{equation}\label{far-field_1}
\begin{split}
&\textrm{$u(x,t) = u_L$ for $x \le -Y$, and  $u(x,t) = u_R$ for $x \ge Y$}, \quad t \in [0,T],\\
&Y = X + 2T \max(L_f,L_g),
\end{split}
\end{equation}
and the same relation holds for $u^{\d}$, with the same $Y$.

We will employ the following notation for $\abs{j}>0$: 
\begin{equation}
\D_{\pm}^j Z_j = 
\begin{cases}
\D_+ Z_j, \quad & j>0, \\
\D_- Z_j, \quad & j<0.
\end{cases}
\end{equation}
\begin{lemma}\label{lemma_bv_loc}
Fix $r >0$ and define
\begin{equation}\label{def_mK1}
 \mK_1 = 2T \max \left(L_f,L_g \right)\textrm{TV}(u_0) 
 + T \left(\norm{f}_{\infty}+\norm{g}_{\infty} \right).
\end{equation}
 Then for $0 \le n \le N$, and $\D$ sufficiently small,
\begin{equation}\label{bv_loc}
\sum_{\{j:\abs{x_j}>r\}} \abs{\D_{\pm}^j U_j^n} \le \TV(u_0) + 4\mK_1/r, \quad
\sum_{\{j:\abs{x_j}>r\}} \abs{\D_{\pm}^j U_j^{\d,n}} \le \TV(u_0) + 4\mK_1/r.
\end{equation}
\end{lemma}

\begin{proof}
It suffices to prove the first inequality of \eqref{bv_loc}. The proof of the second inequality
is essentially the same.
 For $x_j > r$, we can write
 \eqref{scheme} in incremental form: 
\begin{equation}\label{inc_form}
U_j^{n+1} = U_j^n + C_{\jph}^n \D_+ U_j^n - D_{\jmh}^n \D_- U_j^n,
\end{equation}
where
\begin{equation}
C_{\jph}^n = -\lambda {\barf(U_{j+1}^n,U_j^n)- \barf(U_{j}^n,U_{j}^n)\over U_{j+1}^n-U_j^n}, \quad
D_{\jmh}^n = \lambda {\barf(U_{j}^n,U_j^n)- \barf(U_{j}^n,U_{j-1}^n)\over U_j^n - U_{j-1}^n}.
\end{equation}
From the monotonicity property of the Godunov flux, along with \eqref{barq_lip} and the CFL condition \eqref{CFL}, we have
$C_{\jph}^n, D_{\jmh}^n \in [0,1]$ and $C_{\jph}^n+ D_{\jph}^n \le 1$.
This observation makes it possible to use a slightly modified
version of 
the proof of  \cite[Lemma 4.2]{BGKT},
to derive the following inequality, which is valid for $r>0$ and $\D x$ sufficiently small:
\begin{equation}\label{bvloc_V}
\sum_{x_j>r}\abs {U_{j+1}^n - U_j^n}
\le \TV(u_0)|_{(0,\infty)} + 2\mK_1/r,
\end{equation}
where $\mK_1$ is a bound for $\D x \sum_{n=0}^{N-1} \sumj \abs{U_j^{n+1}-U_j^n}$.
A bound similar to \eqref{bvloc_V}  results by considering $x_j < -r$. Then \eqref{bv_loc} results by combining
the two estimates.

By way of deriving \eqref{def_mK1}, note that due to the fact that the finite volume scheme is monotone
\cite[Lemma 5.1]{KT:VV}, 
\begin{equation}
\D x \sumj \abs{U_j^{n+1}-U_j^n} \le \D x \sumj \abs{U_j^{n}-U_j^{n-1}}\le \cdots \le \D x \sumj \abs{U_j^1 - U_j^0}.
\end{equation}
This yields
\begin{equation}
\begin{split}
\D x \sum_{n=0}^{N-1} \sumj \abs{U_j^{n+1}-U_j^n}
&\le N \D x \sumj \abs{U_j^1 - U_j^0} \\
& = N\D x \lambda \sumj \abs{\D_- h_{\jph}(U_{j+1}^0,U_j^0)}\\
&\le N\D x \left(\sum_{j<0} + \sum_{j>0} \right)\left(C_{\jph}^0 \abs{\D_+ U_j^0} + D^0_{\jmh}\abs{\D_- U_j^0} \right)\\
&+ N \D x \lambda \abs{\bar{f}(U_1^0,U_0^0)) - \bar{g}(U_0^0,U_{-1}^0)}\\
&\le N \D x \lambda L_g \,2 \sum_{j<0} \abs{\D_+ U_j^0}
     + N \D x \lambda L_f \,2 \sum_{j>0} \abs{\D_- U_j^0}\\
&+  N \D x \lambda \left(\norm{f}_{\infty}+ \norm{g}_{\infty} \right)\\
&\le \underbrace{
2T \max\left(L_f,L_g\right) \TV(u_0) + T \left(\norm{f}_{\infty}+ \norm{g}_{\infty} \right)
}_{=\mK_1}.
\end{split}
\end{equation}
The fact that \eqref{bv_loc} also holds for $U_j^{\d,n}$  follows from
the second and third conditions of \eqref{v_init}, along 
with  \eqref{lip_delta}.
\end{proof}

\begin{lemma}\label{lemma_L1_est_1}
Fix $r >0$.
We have for $0 \le n \le N$, and $\D$ sufficiently small,
\begin{equation}\label{U_L1_est}
\D x \sum_{\abs{x_j} > r } \abs{U_j^n - U_j^{\d,n}} \le 
\left(2 \left(Y^{\D} + O(\Dx) \right)\left(\mK_2 +4 \mK_1/r \right)\norm{\hat{U}^n-\hat{U}^{\d,n}}_{\infty}\right)^{1/2},
\end{equation}
where
\begin{equation}\label{def_K2_K3}
\mK_2 =  \umax - \umin + \TV(u_0).
\end{equation}
\end{lemma}

\begin{proof}

Define $\mJ^+ \subset \Z^+$ such that
\begin{equation}
r< \underbrace{\min \{x_j: j \in \mJ^+\}}_{=:x_{j_0}} \le r + \D x, 
\quad Y^{\D} \le \max \{x_{j}: j \in \mJ^+\} < Y^{\D} + \D x,
\end{equation}
and let $\abs{\mJ^+}$ denote the cardinality of $\mJ^+$.
An application of Jensen's inequality and \eqref{U_Uhat_formula}, and using the fact that
$U_j^n - U_j^{\d,n}=0$ for $x_j > Y^{\D}$, yields
\begin{equation}\label{square_1a}
\begin{split}
\left(\Dx \sum_{x_j >r} \abs{U_j^n - U_j^{\d,n}} \right)^2
& \le \abs{\mJ^+} \Dx^2 \sum_{j \in \mJ^+} \left(U_j^n - U_j^{\d,n}\right)^2 \\
&= \abs{\mJ^+} \Dx \sum_{j \in \mJ^+}  \left(U_j^n - U_j^{\d,n}\right) \D_- \left(\hat{U}_j^n - \hat{U}_j^{\d,n}\right).
\end{split}
\end{equation}
Summing by parts on the right side of \eqref{square_1a} and using
the fact that $U_j^n  = U^{\d,n}_j $ for $x_j \ge Y^{\D}$, we find that
\begin{equation}\label{square_1}
\begin{split}
\left(\Dx \sum_{x_j>r} \abs{U_j^n - U_j^{\d,n}} \right)^2
&\le -\abs{\mJ^+} \Dx \sum_{j \in \mJ^+} 
             \D_+\left(U_j^n - U_j^{\d,n}\right) \left(\hat{U}_{j}^n - \hat{U}_{j}^{\d,n}\right)\\
             &\quad -\abs{\mJ^+} \Dx\left(U_{j_0}^n - U_{j_0}^{\d,n}\right)\left(\hat{U}_{j_0 -1}^n - \hat{U}_{j_0 -1}^{\d,n} \right) 
\\
&\le \left(Y^{\D} + O(\Dx) \right) \norm{\hat{U}^n-\hat{U}^{\d,n}}_{\infty} \sum_{j \in \mJ^+} \abs{\D_+\left(U_j^n - U_j^{\d,n}\right)}\\
&\quad +\left(Y^{\D} + O(\Dx) \right) (\umax - \umin) \norm{\hat{U}^n-\hat{U}^{\d,n}}_{\infty}\\
&= \left(Y^{\D} + O(\Dx) \right) \norm{\hat{U}^n-\hat{U}^{\d,n}}_{\infty}
\left\{
\sum_{x_j>r}\abs{\D_+\left(U_j^n - U_j^{\d,n}\right) + (\umax-\umin)}
\right\}
\end{split}
\end{equation}

\noindent
Combining \eqref{square_1} with a similar estimate for $x_j <-r$, the result is
\begin{equation}\label{square_1a}
\begin{split}
\left(\Dx \sum_{\abs{x_j}>r} \abs{U_j^n - U_j^{\d,n}} \right)^2
\le 
\left(Y^{\D} + O(\Dx) \right) \norm{\hat{U}^n-\hat{U}^{\d,n}}_{\infty}
\left\{
\sum_{\abs{x_j}>r}\abs{\D^j_{\pm}\left(U_j^n - U_j^{\d,n}\right) + 2(\umax-\umin)}
\right\}.
\end{split}
\end{equation}

\noindent
By Lemma~\ref{lemma_bv_loc} (applied to $U_j^n$ and $U_j^{\d,n}$), along with the triangle inequality,
we have
\begin{equation}\label{square_2}
\begin{split}
 \sum_{\abs{x_j}>r} \abs{\D^j_{\pm}\left(U_j^n - U_j^{\d,n}\right)} 
 &\le \sum_{\abs{x_j}>r} \abs{\D^j_{\pm}U_j^n} 
  +\sum_{\abs{x_j}>r} \abs{\D^j_{\pm} U_j^{\d,n}}  \\
&\le 2\left(\TV(u_0)  +4\mK_1/r\right).
\end{split}
\end{equation}

\noindent
Substituting this into \eqref{square_1a} yields
\begin{equation}
\left( \sum_{\abs{x_j}>r} \abs{\D^j_{\pm}\left(U_j^n - U_j^{\d,n}\right)} \right)^2
 \le
 \left(Y^{\D} + O(\Dx) \right) \norm{\hat{U}^n-\hat{U}^{\d,n}}_{\infty}
 \left\{2\left(\TV(u_0)  +4\mK_1/r\right) + 2(\umax - \umin) \right\}.
\end{equation}

\noindent
Then taking the square root of both sides, and recalling the definition of $\mK_2$,
we obtain \eqref{U_L1_est}.
\end{proof}


\begin{lemma}\label{lemma_L1_est_4}
Fix $r>0$ and $t \in [0,T]$. We have
\begin{equation}\label{u_ft_est}
\begin{split}
\norm{u(\cdot,t)-u^{\d}(\cdot,t)}_{L^1(\R)}
&\le \left(2 Y\left\{\mK_2 +4 \mK_1/r \right\} \norm{\hat{u}(\cdot,t) - \hat{u}^{\d}(\cdot,t)}_{L^{\infty}(\R)}\right)^{1/2}\\
&+ 2 \left(\umax - \umin \right)r, \quad t \in [0,T].
\end{split}
\end{equation}
\end{lemma}

\begin{proof}
Fix $r>0$. Let $\mI = \mI(r,\D) = \{j \in \Z: \abs{x_j} \le r\}$, and
let $\abs{\mI}$ denote the cardinality of $\mI$.
Note that $ \D x \abs{\mI} \le 2(r + O(\D x))$.
Then
\begin{equation}\label{est_near_zero}
\D x \sum_{\abs{x_j}\le r}  \abs{U_j^n - U_j^{\d,n}}
\le \D x \abs{\mI} \left(\umax - \umin \right) \le 2 \left(\umax - \umin \right)(r + O(\D x)).
\end{equation}
Combining \eqref{est_near_zero} with \eqref{U_L1_est} of Lemma~\ref{lemma_L1_est_1} yields
\begin{equation}\label{U_ft_est_disc}
\begin{split}
\D x \sum_{j \in \Z}  \abs{U_j^n - U_j^{\d,n}} \le 
&\left(2 \left(Y^{\D} + O(\Dx) \right)\left\{\mK_2 +4 \mK_1/r \right\}\norm{\hat{U}^n-\hat{U}^{\d,n}}_{\infty}\right)^{1/2}\\
&+2 \left(\umax - \umin \right)(r + \D x).
\end{split}
\end{equation}

For $t \in I^n$,
\begin{equation}
\begin{split}
&\D x \sum_{j \in \Z}  \abs{U_j^n - U_j^{\d,n}} 
	= \norm{u^{\Delta}(\cdot,t)-u^{\d,\Delta}(\cdot,t)}_{L^1(\R)}, \\
&\norm{\hat{U}^n-\hat{U}^{\d,n}}_{\infty}
	= \norm{\hat{u}^{\Delta}(\cdot,t) - \hat{u}^{\delta,\Delta}(\cdot,t)}_{L^{\infty}(\R)}.
\end{split}
\end{equation}

\noindent
Substituting this into \eqref{U_ft_est_disc} results in
\begin{equation}\label{U_ft_est_disc_1}
\begin{split}
\norm{u^{\Delta}(\cdot,t)-u^{\d,\Delta}(\cdot,t)}_{L^1(\R)}
 \le 
&\left(2 \left(Y^{\D} + 
O(\Dx) \right)\left\{\mK_2 +4 \mK_1/r \right\} \norm{\hat{u}^{\Delta}(\cdot,t) - \hat{u}^{\delta,\Delta}(\cdot,t)}_{L^{\infty}(\R)} \right)^{1/2}\\
&+2 \left(\umax - \umin \right)(r + \D x).
\end{split}
\end{equation}

\noindent
On the left side of \eqref{U_ft_est_disc_1} we use the triangle inequality in the form
\begin{equation}\label{tri_1}
\norm{u^{\D} - u^{\d,\D}}_{L^1} \ge \norm{u - u^{\d}}_{L^1} -  \norm{u - u^{\D}}_{L^1} -  \norm{u^{\d,\D} - u^{\d}}_{L^1},
\end{equation}
and on the right side we use
\begin{equation}\label{tri_2}
\norm{\hat{u}^{\D} - \hat{u}^{\d,\D}}_{L^{\infty}} \le \norm{\hat{u}^{\D}-\hat{u}}_{L^{\infty}} 
+ \norm{\hat{u}-\hat{u}^{\d}}_{L^{\infty}} + \norm{\hat{u}^{\d}-\hat{u}^{\d,\D}}_{L^{\infty}}.
\end{equation}
Then
the inequality \eqref{u_ft_est} follows from \eqref{U_ft_est_disc_1} by sending
$\D \rightarrow 0$ and invoking
Remark~\ref{remark_conv}.

\end{proof}

It remains to estimate the term $\norm{\hat{u}(\cdot,t) - \hat{u}^{\d}(\cdot,t)}_{L^{\infty}(\R)}$
appearing in \eqref{u_ft_est} of Lemma~\ref{lemma_L1_est_4}.
For the time being we assume that the front tracking initial data $u_0^{\d}$ satisfies some restrictions,
beyond \eqref{v_init}.
Ultimately we will prove that the advertised rates of convergence hold without the restrictions, i.e.,
the weaker \eqref{v_init} is sufficient.
To specify the restricted $u_0^{\d}$, we start by defining
\begin{equation}\label{restricted_1}
\textrm{
$u^{\d}_0(x) = u_L$ for $x< -X$, 
and $u^{\d}_0(x) = u_R$ for $x \ge X$,
}
\end{equation}
as in the original specification \eqref{v_init}.
Thus it suffices to specify $u^{\d}(x)$ for $x \in [-X,X)$.
To this end choose a positive integer $\mM$ and
any partition $\{z_0, \ldots, z_{\mM} \}$ of $[-X,X]$ such that 
\begin{subequations}\label{partition_d}
\begin{equation}\label{partition_d_a}
-X = z_0 < z_1 < \cdots <z_{\mM} =X,
\end{equation}
\begin{equation}\label{partition_d_b}
\textrm{$\D z_i := z_i - z_{i-1} \le \d$ for $i = 1, \ldots, \mM$,}
\end{equation}
\begin{equation}\label{partition_d_c}
\TV(u_0)|_{(z_{i-1},z_i)} \le \d.
\end{equation}
\end{subequations}

\begin{lemma}
A partition of $[-X,X]$ satisfying \eqref{partition_d} exists.
\end{lemma}

\begin{proof}
Define the partition $\{y_0,\ldots,y_m\}$ of $[-X,X]$
by
\begin{equation}\label{part_2}
y_{i}-y_{i-1} = \d, \quad i=1,\ldots, y_{m-1}, \quad y_m-y_{m-1} \le \d,
\end{equation}
i.e., the partition points are equally spaced, except possibly at the right endpoint $X$.

We now define another
partition $\{\xi_0, \xi_1, \ldots, \xi_M \}$ of $[-X,X]$.
Recalling that $u_0\in BV(\R)$, we can assume that $u_0$ is right continuous, and employ the function $\Lambda(x):=TV(u_0,(-\infty,x])$ for $x\in\R$. Note that $\Lambda$ is right continuous, non-decreasing, and $\Lambda(X) = \TV(u_0)< \infty$. 
The partition points $\xi_{\mu}$ are defined recursively:
\begin{equation}\label{def:xj}
\xi_0 = -X, \quad
\xi_\mu:=\inf \underbrace{\{x \ge \xi_{\mu-1}:\,\La(x)-\La(\xi_{\mu-1})\geq \de\}}_{=:\mathcal{S}_{\mu}}\mbox{ for } \mu \ge 1.
\end{equation}
The construction ends when for some $\mu = M$,
$\mathcal{S}_M = \emptyset$ or $\xi_M \ge X$, in which case we set $\xi_M = X$.
$M$ is guaranteed to be finite, in fact $M \le 1 + \TV(u_0)/\d$.

Finally, take the common refinement of the two partitions $\{\xi_0, \ldots, \xi_M \}$
and $\{y_0, \ldots y_m \}$, resulting in a partition
$\{z_0, \ldots, z_{\mM} \}$ satisfying \eqref{partition_d}.

\end{proof}

Still assuming that \eqref{restricted_1} and \eqref{partition_d} hold,
let $\mC_0$ denote the characteristic function of $(-\infty,-X)$,
$\mC_{\mM+1}$ the characteristic function of $[X,\infty)$, and
$\mC_i$ the characteristic function of $[z_{i-1},z_i)$ for
$i = 1, \ldots,\mM$.
Then $u_0^{\d}$ is defined to be the piecewise constant function
\begin{equation}\label{def:u-delta}
u^{\d}_0(x) = \sum_{i=0}^{\mM+1}\mC_i(x)\bar{u}_0^i,
\quad
\bar{u}_0^i =
\begin{cases}
{1 \over \D z_i}\int_{z_{i-1}}^{z_i}u_0(x)\,dx, &\quad i = 1, \ldots, \mM,\\
u_L, &\quad i = 0,\\
u_R, &\quad i = \mM+1.
\end{cases}
\end{equation}

\begin{lemma}
Front tracking initial data $u_0^{\d}$ satisfying the specifications \eqref{restricted_1}, \eqref{partition_d},
\eqref{def:u-delta} also satisfies \eqref{v_init}.
\end{lemma}

\begin{proof}
Only the requirements $\TV(u_0^{\d}) \le \TV(u_0)$ 
and $\norm{u_0^{\d}-u_0}_{L^1(\R)} = O(\d)$ as $\d \rightarrow 0$ 
require verification. The first requirement is satisfied due to the fact that $u_0^{\d}$
is constructed using cell averages of $u_0$.

For the second requirement,
\begin{equation}
\begin{split}
\norm{u_0^{\d}-u_0}_{L^1(\R)}
&= \int_{-X}^X \abs{u_0^{\d}(x) - u_0(x)}\,dx\\
&\le \sum_{i+1}^\M \int_{\mC_i} {1 \over \D z_i} \int_{\mC_i} \abs{u_0(\xi)-u_0(x)} \, d\xi \, dx\\
&\le \sum_{i+1}^\M \int_{\mC_i} {1 \over \D z_i} \int_{\mC_i} \TV(u_0)|_{\mC_i} \, d\xi \, dx\\
&\le 2X \d.
\end{split}
\end{equation}
\end{proof}

\begin{lemma}\label{lemma_init_data_error}
Assume that $f^\d$ and $g^\d$ are constructed as specified in Section~\ref{sec_intro}, and
that $u_0^\d$ is constructed according to \eqref{restricted_1}, \eqref{partition_d},
\eqref{def:u-delta}. Then
 \begin{equation}\label{f_diff}
\norm{f-f^{\d}}_{\infty} \le {1 \over 8} \norm{f''}_{\infty} \d^2,
 \quad
\norm{g-g^{\d}}_{\infty} \le {1 \over 8} \norm{g''}_{\infty} \d^2,
 \end{equation}
\begin{equation}\label{init_data_error}
\norm{\hat{u}_0 - \hat{u}_0^{\d}}_{\Linf(\R)} \le   \d^2.
\end{equation}
\end{lemma}

\begin{remark}\normalfont
The purpose of (temporarily) requiring the more restrictive specification for
the front tracking initial data given by
 \eqref{restricted_1}, \eqref{partition_d},
\eqref{def:u-delta}
is to provide the estimate \eqref{init_data_error}.
\end{remark}

\begin{proof}
Recalling Assumption~\ref{fg_C2}, the estimate \eqref{f_diff} is a standard result about piecewise linear interpolation.
For the proof of \eqref{init_data_error}, a straightforward calculation employing 
\eqref{partition_d}, \eqref{def:u-delta} and
induction on $i\in \{0, \ldots, \mM \}$
reveals that
\begin{equation}\label{u0_5_a}
\hat{u}_0(z_i) = \hat{u}_0^{\d}(z_i), \quad i=0,\ldots \mM.
\end{equation}
Next, we claim that
\begin{equation}\label{u0_5}
\textrm{if $x \in (z_{i-1},z_{i})$ for some $i \in \{1,\ldots, \mM\}$, then
$\abs{\hat{u}_0(x) - \hat{u}^{\d}_0(x)} \le   \d^2$}.
\end{equation}
To prove  \eqref{u0_5} we use \eqref{u0_5_a} and \eqref{hat_def_integrals}:
\begin{equation}\label{u0_4}
\begin{split}
\abs{\hat{u}_0(x) - \hat{u}^{\d}_0(x)}
&= \abs{\left(\hat{u}_0(x) - \hat{u}_0(z_{i-1})\right) - \left(\hat{u}^{\d}_0(x)- \hat{u}^{\d}_0(z_{i-1})\right)}\\
&= \abs{\int_{z_{i-1}}^x \left(u_0(y) - u_0^{\d}(y) \right) \,dy}\\
&\le \int_{z_{i-1}}^{z_i} \abs{u_0(y) - \bar{u}_0^i} \,dy\\
&\le (z_i - z_{i-1}) \TV(u_0)|_{(z_{i-1},z_i)}.
\end{split}
\end{equation}
Recalling \eqref{partition_d}, we have \eqref{u0_5}.
The proof is completed by combining \eqref{u0_5} with \eqref{u0_4},
along with the fact that $u_0$ and $u_0^{\d}$ take the same constant values for
$\abs{x}\ge X$.
\end{proof}

\begin{theorem}[Main theorem]\label{theorem_rate}
The front tracking algorithm described in Section~\ref{sec_intro} 
converges at a rate no less than $\d^{1/2}$ as $\d \rightarrow 0$. More specifically,
for $\d$ sufficiently small and $t \in [0,T]$,
\begin{equation}\label{theorem_final_estimate}
\begin{split}
\norm{u(\cdot,t)-u^{\d}(\cdot,t)}_{L^1(\R)}
&\le \left(2 Y \mC_1 \left\{\mK_2 \delta^2 +4 \mK_1 \d \right\} \right)^{1/2}
+2 \left(\umax - \umin \right)\d\\
& + \d \TV(u_0)+\norm{u_0 - u_0^{\d}}_{L^1(\R)}, 
\end{split}
\end{equation}
where
\begin{equation}
\begin{split}
&Y = X + 2T\max (L_f,L_g),\quad
\mC_1 =1 + {1 \over 8} T \max \left( \norm{f''}_{\infty},  \norm{g''}_{\infty} \right),\\
& \mK_1 = 2T \max \left(L_f,L_g \right)\textrm{TV}(u_0) 
 + T \left(\norm{f}_{\infty}+\norm{g}_{\infty} \right), \quad
 \mK_2 =  \umax - \umin + \TV(u_0).
\end{split}
\end{equation}
If the front tracking initial data is constructed according to 
 \eqref{restricted_1}, \eqref{partition_d},
\eqref{def:u-delta}, then
 \eqref{theorem_final_estimate} holds without the final two terms on the right side.
\end{theorem}

\begin{remark}\normalfont
Concerning the final term on the right side of \eqref{theorem_final_estimate},
recall that according to \eqref{v_init} we are assuming that $\norm{u_0 - u_0^{\d}}_{L^1(\R)} = O(\d)$.
\end{remark}

\begin{proof}
For now we are still assuming that the initial data $u_0^{\d}$ 
is constructed according to  \eqref{restricted_1}, \eqref{partition_d},
\eqref{def:u-delta}.
From Lemma~\ref{lemma_init_data_error} and Lemma~\ref{lemma_uhat_diff} we obtain
\begin{equation}\label{define_C1}
\norm{\hat{u}(\cdot,t) - \hat{u}^{\d}(\cdot,t)}_{L^\infty(\R)} 
\le \underbrace{\left(1 + {1 \over 8} T \max \left( \norm{f''}_{\infty},  \norm{g''}_{\infty} \right) \right)}_{=:\mC_1}\d^2.
\end{equation}
Plugging this into \eqref{u_ft_est} of Lemma~\ref{lemma_L1_est_4}, we obtain for each $r >0$, and $t \in [0,T]$,
\begin{equation}
\begin{split}
\norm{u(\cdot,t)-u^{\d}(\cdot,t)}_{L^1(\R)} \le 
&\left(2 Y\left\{\mK_2 +4 \mK_1/r \right\}\mC_1 \d^2\right)^{1/2}\\
&+2 \left(\umax - \umin \right)r.
\end{split}
\end{equation}

\noindent
We choose $r=\d$, which yields 
\begin{equation}\label{theorem_final_estimate_A}
\begin{split}
\norm{u(\cdot,t)-u^{\d}(\cdot,t)}_{L^1(\R)}
&\le \left(2 Y \mC_1 \left\{\mK_2 \delta^2 +4 \mK_1 \d \right\} \right)^{1/2}
+2 \left(\umax - \umin \right)\d, 
\end{split}
\end{equation}
and completes the proof when the initial data $u_0^{\d}$
is constructed according to 
 \eqref{restricted_1}, \eqref{partition_d},
\eqref{def:u-delta}.

Now let $v_0^{\d}$ denote
front tracking initial data constructed according to the original specification \eqref{v_init},
and let $v^{\d}$ denote the resulting front tracking solution. Thus $v^{\d}$ and $u^{\d}$ are
constructed using the same fluxes $f^{\d}$ and $g^{\d}$; only the initial data is different.
Using \eqref{partition_d_b} and \eqref{def:u-delta},
it is readily verified that
\begin{equation}\label{u0_ineq_1}
\norm{u_0^{\d} - u_0}_{L^1(\R)} \le \d \TV(u_0).
\end{equation}
An application of the triangle inequality combined with \eqref{u0_ineq_1}  then yields
\begin{equation}\label{ud_vd}
\norm{u_0^{\d} - v_0^{\d}}_{L^1(\R)} \le  \d \TV(u_0) + \norm{u_0 - v_0^{\d}}_{L^1(\R)}.
\end{equation}
Using  the triangle inequality, Theorem~\ref{thm:unique}, and \eqref{ud_vd}, we
find that
\begin{equation}\label{ft_est_final}
\begin{split}
\norm{u(\cdot,t) - v^{\d}(\cdot,t)}_{L^1(\R)}
&\le \norm{u(\cdot,t) - u^{\d}(\cdot,t)}_{L^1(\R)} + \norm{u^{\d}(\cdot,t) - v^{\d}(\cdot,t)}_{L^1(\R)}\\
&\le \norm{u(\cdot,t) - u^{\d}(\cdot,t)}_{L^1(\R)} + \norm{u^{\d}_0 - v^{\d}_0}_{L^1(\R)}\\
&\le \norm{u(\cdot,t) - u^{\d}(\cdot,t)}_{L^1(\R)} +  \d \TV(u_0) + \norm{u_0 - v_0^{\d}}_{L^1(\R)}.
\end{split}
\end{equation}
Substituting \eqref{theorem_final_estimate_A} into \eqref{ft_est_final}, and then renaming $v^{\d} = u^{\d}$,
$v_0^{\d} = u_0^{\d}$, the proof is complete.
\end{proof}



\section{\bf The case where a spatial variation bound is available}\label{sec_conv_bv}
In this section we make the following additional assumption:
\begin{assumption}\label{u_bv}
For some $\mK_3$ which is independent of $\d$ and $\D$,
\begin{equation}\label{u_bv_a}
\TV(u^{\D}(\cdot,t)), \TV(u^{\d,\D}(\cdot,t)),\TV(u(\cdot,t)), \TV(u^{\d}(\cdot,t))\le \mK_3, \quad t \in [0,T].
\end{equation}
\end{assumption}

\begin{remark}\normalfont
With Assumption~\ref{u_bv}, we have existence of traces for $u$ without
the linear non-degeneracy assumption \ref{nonlin_deg}, i.e., assumption
\ref{nonlin_deg} is not needed in this case.
\end{remark}

The following lemma plays the role of Lemma~\ref{lemma_L1_est_1} in this simplified setting.
\begin{lemma}\label{lemma_L1_est_3_a}
With the addition of Assumption~\ref{u_bv}, we have
\begin{equation}\label{L1_est_1A}
\D x \sum_{j \in \Z} \abs{U_j^n - U_j^{\d,n}} \le 
\left(4\left(Y^{\D}  + O(\D x)\right)\mK_3 \norm{\hat{U}^n - \hat{U}^{\d,n}}_{\infty} \right)^{1/2},
\quad n=0, \ldots,N.
\end{equation}
\end{lemma}

\begin{proof}
Let $\mI = \{j \in \Z : -(Y^{\D}+\D x) \le x_j \le Y^{\D}+\D x \}$. Following the relevant parts of the 
proof of Lemma~\ref{lemma_L1_est_1}, we find that
\begin{equation}
\left(\Dx \sum_{j \in \mI} \abs{U_j^n - U_j^{\d,n}} \right)^2
\le 2 \left(Y^{\D} + O(\D x) \right) \norm{\hat{U}^n - \hat{U}^{\d,n}}_{\infty} 
	\sum_{j \in \mI} \abs{\D_+ \left(U_j^n - U_j^{\d,n} \right)}.
\end{equation}
Employing Assumption~\ref{u_bv} and \eqref{far-field}, along with the 
triangle inequality, we find that
\begin{equation}
\left(\Dx \sum_{j \in \Z} \abs{U_j^n - U_j^{\d,n}}\right)^2
\le 4 \left(Y^{\D}+O(\D x) \right) \mK_3 \norm{\hat{U}^n - \hat{U}^{\d,n}}_{\infty},
\end{equation}
from which \eqref{L1_est_1A} is evident.
\end{proof}

The counterpart of Theorem~\ref{theorem_rate} in this setting is the following theorem.
\begin{theorem}\label{theorem_rate_bv_data}
In addition to the previously stated assumptions, 
assume that Assumption~\ref{u_bv} holds.
We have the following error estimate for the front tracking
approximations:
\begin{equation}\label{ft_err_with_bv}
\norm{u(\cdot,t)-u^{\d}(\cdot,t)}_{L^1(\R)} 
\le 2\left( Y \mK_3 \mC_1 \right)^{1/2} \d
 + \d \TV(u_0)+\norm{u_0 - u_0^{\d}}_{L^1(\R)}, \quad t \in [0,T].
\end{equation}
If the front tracking initial data is constructed according to 
 \eqref{restricted_1}, \eqref{partition_d},
\eqref{def:u-delta}, then
 \eqref{ft_err_with_bv} holds without the final two terms on the right side.
\end{theorem}

\begin{proof}
We initially assume that
the front tracking initial data is constructed according to \eqref{restricted_1}, \eqref{partition_d},
\eqref{def:u-delta}.
The estimate \eqref{L1_est_1A} can be written in the form
\begin{equation}
\norm{u^{\D}(\cdot,t)-u^{\d,\D}(\cdot,t)}_{L^1(\R}
\le \left(4\left(Y^{\D}  + O(\D x)\right)\mK_3 \norm{\hat{u}^{\D}(\cdot,t) - \hat{u}^{\d,\D}(\cdot,t)}_{L^\infty(\R)} \right)^{1/2}.
\end{equation}

\noindent
Recalling the triangle inequalities \eqref{tri_1}, \eqref{tri_2}, then sending $\D \rightarrow 0$ and invoking Remark~\ref{remark_conv} yields
\begin{equation}
\norm{u(\cdot,t)-u^{\d}(\cdot,t)}_{L^1(\R}
\le \left(4 Y \mK_3 \norm{\hat{u}(\cdot,t) - \hat{u}^{\d}(\cdot,t)}_{L^\infty(\R)} \right)^{1/2}.
\end{equation}
Recalling \eqref{define_C1}, we have
\begin{equation}\label{ft_err_with_bv_1}
\norm{u(\cdot,t)-u^{\d}(\cdot,t)}_{L^1(\R)} 
\le 2\left( Y \mK_3 \mC_1 \right)^{1/2} \d, \quad t \in [0,T].
\end{equation}

To allow for the more general front tracking initial data of \eqref{v_init}, we proceed as in
the proof of Theorem~\ref{theorem_rate}, which results in \eqref{ft_err_with_bv}.

\end{proof}

\subsection{\bf The case where $f$ and $g$ are both increasing or both decreasing}
\label{sec_fg_inc}
In this section we assume that either $f$ and $g$ are both strictly increasing or
both $f$ and $g$ are strictly decreasing.  We show that then the front tracking algorithm
converges like $O(\d)$. This correlates with a recent result of \cite{ruf_ft}, but our method of
analysis is quite different, for example we do not use the Kuznetsov lemma.

For the sake of concreteness we focus on the
case where $f$ and $g$ are both strictly increasing.
Specifically, in addition to the previous assumptions, we assume that for some $\rho >0$,
\begin{equation}\label{fp_gp}
f'(u), g'(u) \ge \rho, \quad u \in (\umin,\umax).
\end{equation}
The counterpart of Theorem~\ref{theorem_rate} in this setting follows.
\begin{theorem}\label{theorem_rate_inc}
Assume that \eqref{fp_gp} holds.
We have the following error estimate for front tracking:
\begin{equation}\label{err_est_inc}
\norm{u(\cdot,t)-u^{\d}(\cdot,t)}_{L^1(\R)}
\le 2\left( Y \mK_3 \mC_1 \right)^{1/2} \d 
+ \d \TV(u_0)+\norm{u_0 - u_0^{\d}}_{L^1(\R)},\quad t \in [0,T],
\end{equation}
where
\begin{equation}
\begin{split}
&Y = X + 2T\max (L_f,L_g),\quad
\mC_1 =1 + {1 \over 8} T \max \left( \norm{f''}_{\infty},  \norm{g''}_{\infty} \right),\\
&\mK_3 =
  + {1 \over \rho}\Bigl(\left(\umax - \umin \right)+\max\left(L_f,L_g\right)  \TV(u_0)
+  \norm{f}_{\infty}+\norm{g}_{\infty} \Bigr).
\end{split}
\end{equation}
If the front tracking initial data is constructed according to 
 \eqref{restricted_1}, \eqref{partition_d},
\eqref{def:u-delta}, then
 \eqref{err_est_inc} holds without the final two terms on the right side.
\end{theorem}

\begin{proof}
It suffices to prove that
\begin{equation}\label{fg_inc_u_bv}
\TV(u^{\D}), \TV(u^{\d,\D}),\TV(u), \TV(u^{\d})\le 
  + {1 \over \rho}\Bigl(
 \left(\umax - \umin \right) +\max\left(L_f,L_g\right)  \TV(u_0)
+  \norm{f}_{\infty}+\norm{g}_{\infty} \Bigr).
\end{equation}
Then we invoke Theorem~\ref{theorem_rate_bv_data} with $\mK_3$ equal to the right side of
\eqref{fg_inc_u_bv}.
To prove \eqref{fg_inc_u_bv}, we first focus on $u^{\D}$. Via a standard calculation for monotone
schemes \cite{CranMaj:Monoton}, we have
\begin{equation} \label{u_bv_1}
\begin{split}
\sumj \abs{U_j^{n+1}-U_j^n} 
&\le  \sumj \abs{U_j^{n}-U_j^{n-1}}  \le \cdots \le \sumj \abs{U_j^{1}-U_j^0}\\
&= \lambda \sum_{j<0} \abs{\D_-\bar{g}(U^0_{j+1},U^0_j)} + 
\lambda \sum_{j>0} \abs{\D_-\bar{f}(U^0_{j+1},U^0_j)} + 
\lambda \abs{\bar{f}(U^0_{1},U^0_0) - \bar{g}(U^0_{0},U^0_{-1}) }\\
&= \lambda \sum_{j<0} \abs{\D_- g(U^0_j)} + 
\lambda \sum_{j>0} \abs{\D_-f(U^0_j)} + 
\lambda \abs{f(U^0_0) - g(U^0_{-1}) }\\
&\le \lambda \max\left(L_f,L_g\right)  \TV(u_0)
+ \lambda \left(\norm{f}_{\infty}+\norm{g}_{\infty} \right).
\end{split}
\end{equation}
Above we have used $\barf(b,a) = f(a)$, $\bar{g}(b,a) = g(a)$, due to $f', g' >0$.
From \eqref{u_bv_1}, along with \eqref{scheme}, it follows that
\begin{equation}\label{u_bv_2}
 \sumj \abs{\D_- h_{\jph}(U_{j+1}^n, U_j^n)}
\le \max\left(L_f,L_g\right)  \TV(u_0)
+  \norm{f}_{\infty}+\norm{g}_{\infty}. 
\end{equation}
On the other hand,
\begin{equation}\label{u_bv_3}
\begin{split}
\sumj \abs{\D_- h_{\jph}(U_{j+1}^n, U_j^n)}
&= \sum_{j<0} \abs{\D_- g(U^n_j)} + 
\sum_{j>0} \abs{\D_-f(U^n_j)} + 
\abs{f(U^n_0) - g(U^n_{-1}) }\\
&\ge \rho \sum_{j<0} \abs{\D_- U^n_j}
+ \rho \sum_{j>0} \abs{\D_-U^n_j}\\
&\ge \rho \sumj \abs{\D_- U^n_j} - \rho \left(\umax - \umin \right).
\end{split}
\end{equation}
Combining \eqref{u_bv_2}, \eqref{u_bv_3}, we have \eqref{fg_inc_u_bv} for $u^{\D}$.
The proof for  $u^{\d,\D}$ is similar. For $u$ and $u^{\d}$ we use the results for
$u^{\D}$ and $u^{\d,\D}$, sending $\D \rightarrow 0$.
\end{proof}

\begin{remark}\normalfont
Theorem~\ref{theorem_rate_inc} holds without items \ref{nonlin_deg} and \ref{finite_crossings} of Assumption~\ref{assumptions_data}. The purpose of item \ref{nonlin_deg} is to ensure the existence of
one-sided traces at $x=0$. In this case the solution, and all of the relevant approximations, have spatial $\BV$ bounds
(see the proof of Theorem~\ref{theorem_rate_inc}), which guarantees the existence of traces. 
Item \ref{finite_crossings} is required in order to 
invoke Theorem 5.4 of \cite{KT:VV},
guaranteeing that the limit of the Godunov approximations
of Section~\ref{Sec:scheme} is a vanishing viscosity solution, specifically that the $\Gamma$ condition is satisfied.
However, in the special case where both $f$ and $g$ are both strictly increasing or strictly decreasing, the $\Gamma$
condition is automatically satisfied for any weak solution, i.e., the $\Gamma$ condition reduces
to the Rankine-Hugoniot condition in this case.
\end{remark}

\subsection{\bf The case where $f=g$}
When there is no flux discontinuity ($f=g$), it is known that the
front tracking algorithm is first order accurate \cite{lucier_moving_mesh}.
Our method of analysis also yields a first order rate of convergence when
$f=g$. 

To address this setup, we must replace Definition~\ref{weak_def} by the more
standard definition of entropy solution for a scalar conservation law; see, e.g.,
\cite{CranMaj:Monoton} or \cite{Holden_Risebro}.
We no longer need Assumptions \ref{nonlin_deg} and \ref{finite_crossings}.
 Our Godunov scheme reduces to a standard monotone
finite volume scheme in this case, and Theorem~\ref{theorem_convergence} holds, with
the limit now being a standard entropy solution.
A version of Theorem~\ref{thm:unique} holds for standard entropy solutions,
\cite{CranMaj:Monoton,Holden_Risebro}. The relevant portions of our analysis
remain valid, with some simplifications.

\begin{theorem}\label{theorem_rate_f_eq_g}
Assume that $f=g$.
We have the following error estimate for the front tracking approximations:
\begin{equation}\label{err_est_f_eq_g}
\norm{u(\cdot,t)-u^{\d}(\cdot,t)}_{L^1(\R)}
\le 2 \left(Y \TV(u_0) \mC_1 \right)^{1/2} \d
+ \d \TV(u_0)+\norm{u_0 - u_0^{\d}}_{L^1(\R)}, \quad t \in [0,T],
\end{equation}
where
\begin{equation}
Y = X + 2T L_f,\quad
\mC_1 =1 + {1 \over 8} T  \norm{f''}_{\infty}.
\end{equation}
If the front tracking initial data is constructed according to 
 \eqref{restricted_1}, \eqref{partition_d},
\eqref{def:u-delta}, then
 \eqref{err_est_f_eq_g} holds without the final two terms on the right side.
\end{theorem}

\begin{proof}
Classical results yield \eqref{u_bv_a} with $\mK_3 = \TV(u_0)$.
The proof then consists of an application of Theorem~\ref{theorem_rate_bv_data} with $\mK_3 = \TV(u_0)$. 
\end{proof}

\end{document}